\documentclass[11pt,reqno,tbtags,a4paper]{amsart}
\usepackage{amssymb}
\usepackage{xpunctuate}
\usepackage{url}
\usepackage[square,numbers]{natbib}
\bibpunct[, ]{[}{]}{;}{n}{,}{,}
\usepackage{graphicx}

\title[the probability that a binomial variable]
{On the probability that a binomial variable is at most its expectation}

\date{2 October, 2020; typos corrected 19 October, 2019}

\author{Svante Janson}
\thanks{Supported by the Knut and Alice Wallenberg Foundation}
\address{Department of Mathematics, Uppsala University, PO Box 480,
SE-751~06 Uppsala, Sweden}
\email{svante.janson@math.uu.se}
\newcommand\urladdrx[1]{{\urladdr{\def~{{\tiny$\sim$}}#1}}}
\urladdrx{http://www2.math.uu.se/~svante/}


\numberwithin{equation}{section}

\renewcommand\le{\leqslant}
\renewcommand\ge{\geqslant}

\allowdisplaybreaks

\theoremstyle{plain}
\newtheorem{theorem}{Theorem}[section]
\newtheorem{lemma}[theorem]{Lemma}

\newtheorem{conjecture}[theorem]{Conjecture}

\theoremstyle{definition}

\newtheorem{exampleqqq}[theorem]{Example}

\newtheorem{remarkqqq}[theorem]{Remark}
\newenvironment{remark}{\begin{remarkqqq}}
  {\hfill\qedsymbol\end{remarkqqq}}

\theoremstyle{remark}

\newenvironment{romenumerate}[1][-10pt]{
\addtolength{\leftmargini}{#1}\begin{enumerate}
 \renewcommand{\labelenumi}{\textup{(\roman{enumi})}}%
 }{\end{enumerate}}

\newcounter{oldenumi}
{\setcounter{oldenumi}{\value{enumi}}
\begin{romenumerate} \setcounter{enumi}{\value{oldenumi}}}
{\end{romenumerate}}

\newcounter{thmenumerate}

\newcounter{xenumerate}   

\newcounter{cases}
\newcounter{subcases}
\newcommand\pfcase[1]{\smallskip\noindent
 \refstepcounter{cases}%
 \setcounter{subcases}{0}%
 \emph{Case \arabic{cases}: #1.} \noindent}
\newcommand\pfsubcase[1]{\smallskip\noindent
 \refstepcounter{subcases}%
 \emph{Case \arabic{cases}\alph{subcases}: #1.} \noindent}

\newcommand{\refT}[1]{Theorem~\ref{#1}}

\newcommand{\refL}[1]{Lemma~\ref{#1}}

\newcommand{\refR}[1]{Remark~\ref{#1}}

\newcommand{\refS}[1]{Section~\ref{#1}}

\newcommand{\refF}[1]{Figure~\ref{#1}}

\newcommand{\refConj}[1]{Conjecture~\ref{#1}}

\newcommand{\sumi}{\sum_{i=1}^\infty}
\newcommand{\sumj}{\sum_{j=1}^\infty}

\newcommand{\sumin}{\sum_{i=1}^n}

\newcommand\set[1]{\ensuremath{\{#1\}}}
\newcommand\bigset[1]{\ensuremath{\bigl\{#1\bigr\}}}

\newcommand\xpar[1]{(#1)}
\newcommand\bigpar[1]{\bigl(#1\bigr)}
\newcommand\Bigpar[1]{\Bigl(#1\Bigr)}

\newcommand\lrpar[1]{\left(#1\right)}

\newcommand\Bigsqpar[1]{\Bigl[#1\Bigr]}

\newcommand\bigabs[1]{\bigl\lvert#1\bigr\rvert}

\def\rompar(#1){\textup(#1\textup)}    
\newcommand\xfrac[2]{#1/#2}

\newcommand\parfrac[2]{\lrpar{\frac{#1}{#2}}}

\def\xexp(#1){e^{#1}}

\newcommand\floor[1]{\lfloor#1\rfloor}

\newcommand\frax[1]{\{#1\}}

\newcommand\ntoo{\ensuremath{{n\to\infty}}}

\newcommand\punkt{\xperiod}    
\newcommand\iid{i.i.d\punkt}    
\newcommand\ie{i.e\punkt}
\newcommand\eg{e.g\punkt}

\newcommand\ii{\mathrm{i}}

\newcommand\bbR{\mathbb R}

\newcommand\bbZ{\mathbb Z}

\newcounter{CC}
\newcommand{\CC}{\stepcounter{CC}\CCx} 
\newcommand{\CCx}{C_{\arabic{CC}}}     
\newcommand{\CCdef}[1]{\xdef#1{\CCx}}     
\newcommand{\CCname}[1]{\CC\CCdef{#1}}    
\newcounter{cc}

\newcommand\E{\operatorname{\mathbb E{}}}
\renewcommand\P{\operatorname{\mathbb P{}}}

\newcommand\Var{\operatorname{Var}}

\newcommand\Po{\operatorname{Po}}
\newcommand\Bi{\operatorname{Bi}}

\newcommand\Be{\operatorname{Be}}

\newcommand\gb{\beta}

\newcommand\gam{\gamma}
\newcommand\gG{\Gamma}

\newcommand\gl{\lambda}
\newcommand\gL{\Lambda}

\newcommand\gs{\sigma}

\newcommand\gss{\sigma^2}

\newcommand\qw{^{-1}}
\newcommand\qww{^{-2}}

\newcommand\qqw{^{-1/2}}

\newcommand\oi{\ensuremath{[0,1]}}

\newcommand\dtv{d_{\mathrm{TV}}}

\newcommand\dd{\,\mathrm{d}}
\newcommand\ddq{\mathrm{d}}
\newcommand\ddd[1]{\frac{\ddq}{\ddq#1}}
\newcommand\ddx{\ddd{x}}
\newcommand\ddxx[1]{\frac{\ddq^{#1}}{\ddq x^{#1}}}
\newcommand\ddxxx[2]{#2^{(#1)}}

\newcommand\rhs{right-hand side}

\newcommand\HH{\Psi}
\newcommand\HHH{\Psi^*}
\newcommand\tHHH{\widetilde{\Psi}^*}
\newcommand\tR{\widetilde{R}}
\newcommand\pox{\varpi}
\newcommand\xki{\frac{k+1}2}
\newcommand\mm{^{(m)}}
\newcommand\Tk{T_1}
\newcommand\sqrppi{\sqrt{2\pi}}
\newcommand\Binp{\Bi(n,p)}
\newcommand\Binm{\Bi(n,m/n)}
\newcommand\jjj{\tfrac{2}{3}}
\newcommand\ppjjj{\bigpar{\jjj}}
\newcommand\Vasek{Va\v sek}
\newcommand\Chvatal{Chv{\'a}tal}

\newcommand{\Holder}{H\"older}

\begin{document}

\begin{abstract} 
Consider the probability that a binomial random variable Bi$(n,m/n)$ with
integer expectation $m$ is at most its expectation.
Chv{\'a}tal conjectured that for any given $n$, this probability is smallest
when $m$ is the integer closest to $2n/3$.
We show that this holds when $n$ is large.
\end{abstract}

\maketitle

\section{Introduction}\label{S:intro}

Consider the probability $\P\bigpar{\Binp\le np}$
that a binomial random variable $\Bi(n,p)$ is
less that or equal to its mean. 
(We slighly abuse notation, and let $\Binp$ denote both the binomial
distribution and a binomial random variable.)
By the central limit theorem, unless $n$ or $p(1-p)$ is small,
this probability is close to $\frac12$; in fact, the Berry--Esseen theorem
\cite{Berry} and \cite{Esseen42} (see also \eg{} \cite[Theorem 7.6.1]{Gut})
shows that $\P(\Binp\le np)=\frac12+O\bigpar{(np(1-p))\qqw}$.
(See also the explicit bounds in 
\cite{Doerr},
\cite{GreenbergM},
\cite{PelekisR},
\cite{RigolletT},
\cite{Slud}.%
)

In the case when $np=m$ is an integer,
\citet{Neumann} showed that the mean $np$ is also a (strong)
median, \ie,
\begin{align}
\P(\Binp<np) < \frac12 < \P(\Binp\le np) ,
\qquad np=m\in\set{0,\dots, n} 
.\end{align}
(See \cite{JogdeoS}, \cite{KaasB}, and 
\cite[Exercise MPR-24]{KnuthIV.5} for other proofs.)
It follows that for any fixed $n\ge1$, the probability $\P\bigpar{\Binp\le np}$
regarded as a function of $p\in\oi$,
oscillates around $\frac12$,  with upward jumps at each $m/n$ and
monotone decrease between the jumps.
See \refF{fig1}
for an example.

\begin{figure}[tbh]
  \centering
\includegraphics[width=1.0\textwidth]
{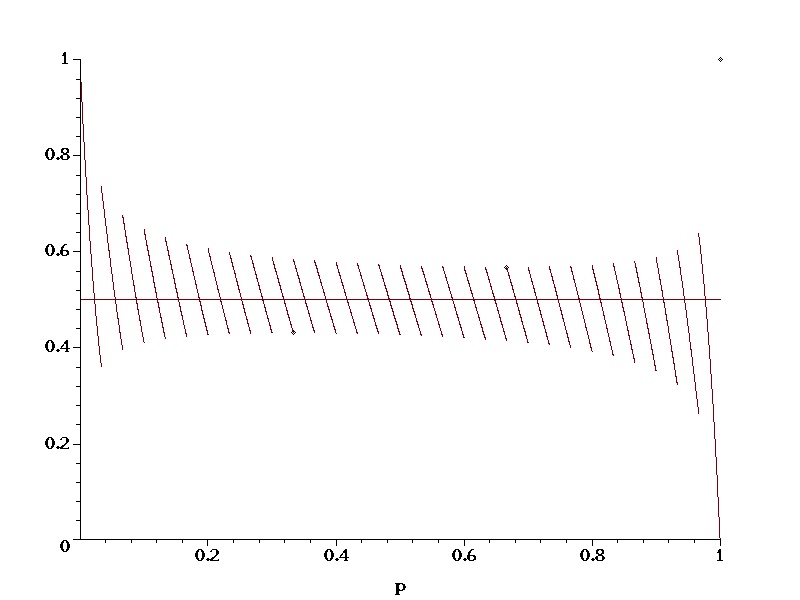}
\caption{
$\P(\Binp\le np)$ for $n=30$.
For integer values of $m=np$,
the  minimum of $\P(\Binp\le np)$ (at $m=20$)
and the maximum of $\P(\Binp< np)$ (at $m=10$)
are marked with dots.
}
 \label{fig1}
\end{figure}

Consider again the case when $np=m$ is an integer, illustrated by the local
maxima in \refF{fig1}.
\Vasek{} \Chvatal{} (personal communication) made the following conjecture,
based on  numerical experiments.

\begin{conjecture}[\Chvatal]\label{Chvatal}
For any fixed $n\ge2$, 
as $m$ ranges over $\set{0,\dots,n}$,
the probability 
$q_m:=\P\bigpar{\Binm\le m}$ is smallest when 
$m$ is the integer closest to
$2n/3$.
\end{conjecture}

The purpose of the present paper is to show that this conjecture
holds for large $n$. Moreover, at least for large $n$,
the probabilities $q_m$ are inverse unimodal, 
\ie, have no other local minimum.
(The latter property was partly proved by \citet[(29)]{RigolletT}, who
proved, for 
any $n$, that $q_m$ decreases for $m\le n/2$. 
We conjecture that also the inverse unimodality holds for all $n$.)

\begin{theorem}
  \label{T1}
There exists $n_0$ such that \refConj{Chvatal} is true for every $n\ge n_0$.
Moreover, still for $n\ge n_0$,
the difference $q_{m+1}-q_m$ is negative when $m+\frac12<2n/3$ and positive when
$m+\frac12>2n/3$. 
\end{theorem}

\begin{remark}
  By symmetry, \ie, considering $n-\Binm$, it follows that for large $n$ at
  least, the probability $\P(\Binm<m)$ is largest for the integer $m$
  closest to $n/3$.
\end{remark}

\begin{remark}\label{Rp}
For general $p$,
  the value of $\P\bigpar{\Binp\le np}$ is asymptotically given by 
  \begin{align}
\P\bigpar{\Binp\le np}
=
\frac12+\frac{4-2p-6\frax{np}}{6\sqrt{2\pi np(1-p)}}+O\parfrac{1}{np(1-p)},    
  \end{align}
at least provided $np(1-p)\ge \log^2n$; this 
is a consequence of \refT{TE3} and \eqref{q1} (with $k=1$).
Cf.\ \refF{fig1}.
See also the explicit related bound by \citet[Lemma 8]{Doerr}.
\end{remark}

\begin{remark}
  In principle, it should be possible to calculate all constants in our
  proof explicitly, and thus find an explicit value for
  $n_0$;  the conjecture then could be verified completely (assuming that it
  holds) by checking all smaller  $n$ numerically. 
However, we do not believe that this is practical.
Presumably, other methods, completely different from ours, 
are needed to show \Chvatal's conjecture
in general.
(We have verified the conjecture numerically for $n\le1000$.)
\end{remark}

Our proof of \refT{T1} is based on 
the version for integer-valued random variables found by \citet{Esseen}
of the asymptotic Edgeworth  expansion 
for probabilities in the central limit theorem.
This is usually stated for a single probability distribution, but we need to
check that the estimates hold uniformly for $\Binp$ with $p$ in some range;
hence we discuss this expansion in some detail in \refS{SE}.
In particular, we state in \refT{TE3} the result that we need in a general form,
and prove it in \refS{SpfTE3}.
We return to the binomial probabilities in \refS{Sex}, and prove \refT{T1}
in \refS{SpfT1}.

\begin{remark}
  See \citet{BCD} for another aspect, with statistical implications, 
of the oscillations of binomial probabilities; 
they too use   Edgeworth  expansions with more than one term.
\end{remark}

\section{Preliminaries}

\subsection{General notation}
We let  $\frax{x}:=x-\floor{x}$ denote the fractional part of a real number
$x$;
thus $\frax{x}\in[0,1)$.

 $\ddxxx mf$ denotes the $m$th derivative of a function $f$, with
$\ddxxx0f:=f$. 

A random variable $X$ has \emph{span} $d$ if it is concentrated on a set 
$x_0+d\bbZ$ for some real $x_0$, and $d$ is maximal with this property.

$C$ and $c$ denote unimportant constants, in general different at different
occurrences. The ``constants'' may depend on some given parameters, given by
the context;
we sometimes write \eg{} $C_k$ to emphasize that $C$ depends on
a parameter $k$, but this is omitted when not necessary.

\subsection{Special notation}
We introduce here the notation needed for the expansions in the following
sections. 
See further \citet[Chapters III--IV]{Esseen} and \citet[\S VI.1]{Petrov}.

Let $X$ be a random variable.
($X$ will be regarded as given. Most quantities defined below depend on the
distribution of $X$, although we for simplicity do not show this in the
notation.)
Denote its mean by $\mu:=\E X$,
its central absolute moments by $\gb_j:=\E|X-\mu|^j$, 
and its cumulants by $\gam_j$ (when they exist, \ie, when $\E|X|^j<\infty$).
Also, let $\gss=\gb_2=\gam_2=\Var X$ be the variance of $X$, and define
the scale-invariant
\begin{align}
\gL_j&:=\frac{\gb_j}{\gs^j}, \label{a0}
\\
  \gl_j&:=\frac{\gam_j}{\gs^j}.\label{a1}
\end{align}
Each cumulant $\gam_j$, $j\ge2$, can be expressed as a polynomial in 
central moments of orders $2,\dots,j$, and it follows, using also \Holder's
inequality, that
\begin{align}\label{b1}
  |\gam_j|\le C_j\gb_j
\end{align}
and thus
\begin{align}\label{b2}
  |\gl_j|\le C_j\gL_j.
\end{align}
Furthermore, \Holder's inequality also easily yields, for $2\le j\le k$,
\begin{align}\label{b3}
  1=\gL_2 \le \gL_j^{1/(j-2)} \le \gL_k^{1/(k-2)}.
\end{align}

Define polynomials $P_j(u)$, $j\ge1$, by expanding the formal power series
\begin{align}\label{a2}
  \exp\lrpar{\sumi \frac{\gl_{i+2}}{(i+2)!}u^{i+2}z^i}
=1+\sumj P_j(u)z^j.
\end{align}
Note that $P_j(u)$ is a polynomial of degree $3j$; moreover,
\begin{align}\label{a3}
  P_j(u)=\sum_{r=j+2}^{3j} \pox_{jr}u^r,
\end{align}
where 
each coefficient $\pox_{jr}$ is a polynomial in $\gl_k$,
$k=3,\dots,j+2$.
In particular, $P_j(u)$ is well-defined provided $\E|X|^{j+2}<\infty$.
Furthermore, 
$\pox_{jr}=0$ unless $r-j$ is even.

Let $\Phi(x)$ be the distribution function of the standard normal
distribution,
so that 
\begin{align}\label{a4}
  \ddx\Phi(x)=\phi(x):=\frac{1}{\sqrt{2\pi}}e^{-x^2/2}.
\end{align}
Define $Q_j(x)$ as the function obtained from $P_j(u)$ by replacing
each power $u^r$ by $(-1)^r \ddxxx{r}\Phi(x)=(-1)^r \ddxx{r}\Phi(x)$,
\ie,
\begin{align}\label{a5a}
  Q_j(x)&:=\sum_{r=j+2}^{3j} (-1)^r\pox_{jr}\ddxxx{r}\Phi(x)
=\sum_{r=j+2}^{3j} (-1)^r\pox_{jr}\ddxxx{r-1}\phi(x)
\\\phantom{:} \label{a5b}
&=-\phi(x)\sum_{r=j+2}^{3j} \pox_{jr} H_{r-1}(x),
\end{align}
where $H_r(x)$ are the Hermite polynomials
(in the normalization natural in probability theory, \ie, orthogonal w.r.t.\
the standard normal distruibution); see \eg{}
\cite[(18.5.9)]{NIST} (there denoted $He_r(x)$)
or \cite[p.~127]{Petrov}.

Define also
periodic functions $\psi_r$, $r\ge1$, by their Fourier series
\begin{align}\label{psi}
  \psi_r(x) 
= \sum_{\substack{k=-\infty\\k\neq 0}}^\infty \frac{e^{2\pi k x \ii}}{(2\pi k \ii)^r}.
\end{align}
Note that for $r\ge2$, the series \eqref{psi} converges absolutely and
defines a continuous periodic function with period 1.
However, for $r=1$, the series is only conditionally convergent; in fact it
is the Fourier series of $\frac12-\frax{x}$.
(It follows from standard results that the series converges for every
$x$, but we do not really need this.) 
Hence, $\psi_1(x)$ has a jump 1 at every integer.
For later convenience, we redefine
$\psi_1$ to be right-continuous; thus we define
\begin{align}\label{psi1}
  \psi_1(x):=\frac12-\frax{x}=\frac12-x+\floor{x}, \qquad x\in\bbR,
\end{align}
noting that for $r=1$, \eqref{psi} holds only for non-integer $x$.
Note also that, for any $r\ge1$ and $x\in\bbR$,
\begin{align}\label{psib}
  \psi_r(x) =- \frac{1}{r!} B_r(\frax{x}),
\end{align}
where $B_r(x)$ denotes the Bernoulli polynomials, see \eg{}
\cite[(24.8.3)]{NIST}.
In particular, \cite[(24.2.4)]{NIST},
\begin{align}\label{ber}
  \psi_r(0) =- \frac{1}{r!} B_r(0)
=- \frac{1}{r!} B_r,
\end{align}
where $B_r$ denotes the Bernoulli numbers.
Recall that $B_1=-\frac12$, $B_2=\frac{1}6$ and $B_{2j+1}=0$ for $j\ge1$.

\begin{remark}
Note that $\psi_r(x) = (-1)^{r-1}h_r Q_r(x)$ (where $h_r=\pm1$)
in the notation of
\cite[p.~60--61]{Esseen}.
We prefer the choice of signs in our definition, but this is only a matter
of taste.   
\end{remark}

\section{The basic expansion theorem}\label{SE}

\citet[Theorem IV.4]{Esseen} proved
the following.
More precisely, 
Esseen's result is \eqref{ra};
the version \eqref{rb} follows immediately by applying \eqref{ra} with $k+1$
instead of $k$. (We prefer this, weaker, version for our generalization in
\refT{TE3}.) 

\begin{theorem}[\citet{Esseen}]\label{TE1}
  Let $X,X_1,X_2,\dots$ be \iid{} integer valued
random variables with span $1$ and let $S_n:=\sumin X_i$.
Let $k\ge1$ be an integer and suppose that $\E|X|^{k+2}<\infty$.
With notations as above, define
\begin{align}\label{hnk}
\HH_{n,k}(x)
:= \Phi(x) + \sum_{j=1}^k n^{-j/2}Q_j(x)
\end{align}
and
\begin{align}\label{cge}
\HHH_{n,k}(x)
:= 
\HH_{n,k}(x)+\sum_{\ell=1}^k(-1)^{\ell-1}n^{-\ell/2}\gs^{-\ell}
  \psi_\ell\bigpar{n\mu+x\gs\sqrt n}\ddxxx{\ell}{\HH_{n,k}}(x)
.\end{align}
Then
\begin{align}\label{esseen}
&  \P\bigpar{S_n\le n\mu + x\gs\sqrt n}
= \HHH_{n,k}(x)
+ R_{n,k}(x)
,\end{align}
where, as \ntoo,
\begin{align}\label{ra}
  \sup_x|R_{n,k}(x)| = o\bigpar{n^{-k/2}}.
\end{align}

If further $\E |X|^{k+3}<\infty$, then 
\begin{align}\label{rb}
  \sup_x|R_{n,k}(x)| = O\bigpar{n^{-\frac{k+1}{2}}}.
\end{align}
\end{theorem}

\refT{TE1} is stated for a single distribution.
We want to apply it to $X\sim\Be(p)$, but then need some uniform estimates for
all $p$, or at least for a large range.
It is no surprise that the proof of \refT{TE1} yields such uniformity under
suitable conditions, including some uniform moment estimates. 
For $\Be(p)$, 
the case $p\in[p_1,p_2]$ for a compact interval $[p_1,p_2]\subset(0,1)$
does not cause any difficulties, but we can go beyond that.
We will show the following extension of \refT{TE1} in \refS{SpfTE3}.

\begin{theorem}\label{TE3}
Let $k\ge1$.  
For every $c>0$, there exists a constant $C$ (depending on $k$ and $c$ only)
such that  if $X$ is an integer-valued random variable with
$\E|X|^{k+3}<\infty$ and
\begin{align}\label{te3a}
  \min\bigset{\P(X=0),\P(X=1)}&
\ge c \gss,
\end{align}
and  $n\ge2$ is an integer with
\begin{align}\label{te3bad}
\gs\sqrt n\ge \log n,  
\end{align}
then
\eqref{esseen} holds with
\begin{align}\label{te3b}
\sup_x
 |R_{n,k}(x)| 
\le C \gL_{k+3}n^{-\xki}
\le C  \frac{\E|X|^{k+3}}{\gs^{k+3}} n^{-\xki}
.\end{align}
\end{theorem}

\begin{remark}
  We defined $\psi_1(x)$ to be right-continuous so that $\HHH_{n,k}(x)$ and
  $R_{n,k}(x)$ are right-continuous, which enables \eqref{ra}, \eqref{rb} and
  \eqref{te3b} to hold for all $x$, including integers.
(All functions in \eqref{hnk}--\eqref{cge} except $\psi_1$ are continuous.)
\end{remark}

\begin{remark}\label{Rsimpler}
As remarked by \citet[p.~61]{Esseen}, \eqref{cge} contains
redundant terms. The form \eqref{cge} is sometimes convenient (for example in
the proof), but 
it is often more convenient to modify \eqref{cge} 
by dropping the redundant terms; we
thus define also
\begin{align}\label{tcge}
\tHHH_{n,k}(x)
:= 
\HH_{n,k}(x)+\sum_{\ell=1}^k(-1)^{\ell-1}n^{-\ell/2}\gs^{-\ell}
  \psi_\ell\bigpar{n\mu+x\gs\sqrt n}\ddxxx{\ell}{\HH_{n,k-\ell}}(x)
\end{align}
and define a modified remainder term $\tR_{n,k}(x)$ by
\begin{align}\label{tesseen}
&  \P\bigpar{S_n\le n\mu + x\gs\sqrt n}
= \tHHH_{n,k}(x)
+ \tR_{n,k}(x)
.\end{align}
It follows from \eqref{hnk} that in \refT{TE1},
this  changes
$\HHH_{n,k}(x)$ and thus $R_{n,k}(x)$ by some terms which are $O(n^{-m/2})$
for some $m\ge k+1$ so \eqref{ra} and \eqref{rb} still hold for
$\tR_{n,k}(x)$.
With only a little more effort, it can be verified that the same holds for
\refT{TE3}; it follows from \refL{L1} below that
the removed terms are all dominated by $\gL_{k+3}n^{-\xki}$, so \eqref{te3b}
still holds. 
(This uses also that  $\gs\sqrt n$ is bounded below by the assumptions,
and the fact 
that $\gb_i\le C_{i,j}\gb_j$ when $2\le i\le j$ and $X$ is integer valued.)
\end{remark}

\begin{remark}
  \label{Rspan}
Note that \refT{T1} is only for random variables with span 1, and that a
uniform version for a family of random variables therefore requires some
uniform condition preventing the variables from being too close to variables
with larger span. We use condition \eqref{te3a}  which is convenient
and turns out to be sufficient; it can obviously be replaced by more general
conditions. 
\end{remark}

\begin{remark}\label{Rannoy}
The assumption \eqref{te3bad} in \refT{TE3} is 
annoying but not a very serious restriction.
Note that the \rhs{} of \eqref{te3b} is, since $X$ is
integer-valued, at least $C(\gs\sqrt n)^{-k-1}$. Hence, if \eqref{te3bad}
is violated, \eqref{te3b} would, even if true, only give a weak bound.
We do not know whether \eqref{te3bad} really is needed.
It is possible that \refT{TE3} could be proved without this assumption,
using the alternative method of proof in \cite[Section IV.4]{Esseen}, but we
have not pursued this. 
\end{remark}

\section{Proof of \refT{TE3}}\label{SpfTE3}

\begin{lemma}\label{L1}
  Suppose that $\ell\ge1$ and\/ $\E|X|^{\ell+2}<\infty$.
Then, for every $m\ge0$, 
\begin{align}
  |\pox_{\ell,r}|&\le C_\ell \gL_{\ell+2},\label{l1a}
\\\label{l1p}
|P_{\ell}(u)|&\le C_\ell \gL_{\ell+2} \bigpar{|u|^{\ell+2}+|u|^{3\ell}},
\\\label{l1q}
|Q_\ell\mm(x)|&\le C_{\ell,m} \gL_{\ell+2},
\\\label{lih}
|\HH_{n,\ell}\mm(x)|&\le C_{\ell,m} \bigpar{1+\gL_{\ell+2}n^{-\ell/2}}
.\end{align}
\end{lemma}

\begin{proof}
  It follows from \eqref{a2} that $\pox_{\ell,r}$ is a linear combination of
  products $\prod_{k=1}^m\gl_{i_k+2}$ with $\sum_k(i_k+2)=r$ and $\sum_k
  i_k=\ell$. 
By \eqref{b2} and \eqref{b3}, each such product is bounded by
\begin{align}
  C_\ell\prod_{k=1}^m\gL_{i_k+2}
\le   C_\ell\prod_{k=1}^m\gL_{\ell+2}^{i_k/\ell}
=C_\ell \gL_{\ell+2},
\end{align}
which yields \eqref{l1a}.
This implies \eqref{l1p} and \eqref{l1q} by \eqref{a3} and \eqref{a5a},
noting that $\Phi(x)$ and all its derivatives are bounded on $\bbR$.
Finally, \eqref{hnk} and \eqref{l1q} together with \eqref{b3} yield
\begin{align}
  |\HH_{n,\ell}\mm(x)|&
\le C_m + C_{\ell,m}\sum_{j=1}^\ell n^{-j/2}\gL_{j+2}
\le  C_{\ell,m}\sum_{j=0}^\ell n^{-j/2}\gL_{\ell+2}^{j/\ell}
\notag\\&
\le  C_{\ell,m}\bigpar{1+ n^{-\ell/2}\gL_{\ell+2}}.
\end{align}
\end{proof}

\begin{lemma}\label{L2}
  If\/ $X$ is a random variable with 
$\P(X=0)\ge a$ and $\P(X=1)\ge a$ for some $a\ge0$,
  then
  \begin{align}\label{l2}
    \bigabs{\E e^{\ii t X}} \le 1-\frac{a}{\pi^2}t^2 \le e^{-a\pi\qww t^2},
\qquad |t|\le\pi.
  \end{align}
\end{lemma}
 \begin{proof}
The assumption implies that $\E e^{\ii tX}-a-ae^{\ii t}$ is the Fourier
    transform of a positive measure with mass $1-2a$. Hence,
if $|t|\le\pi$,
\begin{align}\label{l2a}
    \bigabs{\E e^{\ii t X}} 
&\le 
\bigabs{a+ae^{\ii t}} + \bigabs{\E e^{\ii tX}-a-ae^{\ii t}}
\le
a\bigabs{1+e^{\ii t}}+1-2a
\notag\\&
=2a\bigabs{\cos\frac{t}{2}}+1-2a
=1-2a\bigpar{1-\cos\frac{t}{2}}
=1-4a\sin^2\frac{t}4
,\end{align}
and \eqref{l2} follows.
\end{proof}

\begin{proof}[Proof of \refT{TE3}]
  We follow the proofs of Theorems 3 and 4 in \citet{Esseen}
(with $d=1$ and thus $t_0=2\pi$)
and mention only the main differences.
Note that \cite{Esseen} considers centred variables, so $X_i$ there
is our $X_i-\mu$.
Let 
\begin{align}\label{pf}
f(t)&:=\E e^{\ii t(X-\mu)},
\\\label{fn}
  f_n(t)&:=\E e^{\ii t (S_n-n\mu)/(\gs\sqrt n)}=f^n\bigpar{t/(\gs\sqrt n)},
\\\label{g0}
g_0(t)&:=e^{-t^2/2}+\sum_{j=1}^k n^{-j/2} P_j(\ii t) e^{-t^2/2}.
\end{align}
Also, let $T:=n^{\xki}$ and replace $T_{3n}$ in \cite{Esseen} by 
\begin{align}\label{Tk}
  \Tk:=\sqrt{n}/ \gL_{k+3}^{1/(k+1)}.
\end{align}

The second inequality in \eqref{te3b} is trivial, by the definition
\eqref{a0}.
We may also assume that
$\gL_{k+3}n^{-\xki}\le 1$, and thus $\Tk\ge1$, since otherwise it follows
from \refL{L1} and \eqref{b3} that each term in \eqref{cge}
is bounded by $C \gL_{k+3}n^{-\xki}$, and thus \eqref{te3b} holds
trivially.

In the range $|t|\le T_1$, we have 
\begin{align}\label{cg}
  |f_n(t)-g_0(t)|\le C \gL_{k+3}|t|^{k+3} e^{-t^2/13} n^{-\xki}
\end{align}
by \cite[Lemma VI.4]{Petrov} with $s=k+3$ and $m=0$ 
(which essentially is \cite[Lemma IV.2a]{Esseen} with $T_{3n}$ improved to
our $\Tk/4$,  which can be proved in the same way). 
Hence, for the ``main term'' in the estimate 
\begin{align}\label{cgi}
\int_{-\Tk}^{\Tk} \frac{|f_n(t)-g_0(t)|}{|t|}\dd t\le C \gL_{k+3}n^{-\xki}.
\end{align}
Furthermore, if $| t|\le \pi\gs \sqrt n$,
then the assumption \eqref{te3a} and \refL{L2} yield
\begin{align}\label{mar}
  |f_n(t)|
= \bigabs{f\bigpar{t/(\gs\sqrt n)}}^n
\le 
e^{-nc\gss (t/\gs\sqrt n)^2}=e^{-ct^2}.
\end{align}
Hence,
\begin{align}\label{mai}
  \int_{\Tk}^{\pi\gs\sqrt n} \frac{|f_n(t)|}{|t|}\dd t
\le   \int_{\Tk}^{\pi\gs\sqrt n} {|f_n(t)|}\dd t
\le C e^{-c\Tk^2}
\le C \Tk^{-(k+1)}
=
C \gL_{k+3}n^{-\xki}.
\end{align}
The integral $  \int_{\Tk}^{\pi\gs\sqrt n} \xfrac{|g_0(t)|}{|t|}\dd t$ has the
same estimate.
Consequently, by \eqref{cgi} and \eqref{mai},
\begin{align}\label{cgj}
\int_{-\pi\gs\sqrt n}^{\pi\gs\sqrt n} \frac{|f_n(t)-g_0(t)|}{|t|}\dd t
\le C \gL_{k+3}n^{-\xki}.
\end{align}
The same arguments yield also
\begin{align}\label{cgk}
\int_{-\pi\gs\sqrt n}^{\pi\gs\sqrt n} |f_n(t)-g_0(t)|\dd t\le C \gL_{k+3}n^{-\xki}.
\end{align}
Using the estimates \eqref{cgj} and \eqref{cgk}, the rest of the proof is
essentially the same as in \cite{Esseen}.
One of the terms, generalizing $I'_k$ on \cite[p.~58--59]{Esseen},
is
\begin{align}\label{cgs}
\int_{-\pi\gs\sqrt n}^{\pi\gs\sqrt n}\frac{ |f_n(t)-g_0(t)|}{|t+2\pi j
  \gs\sqrt n|}\dd t\le 
\frac{C}{|j|\gs\sqrt n} \gL_{k+3}n^{-\xki},
\end{align}
where we use \eqref{cgk}. This term exists for all integers $j\neq0$ with
$(2|j|-1)\pi \gs\sqrt n < T=n^{\xki}$, and thus the sum of them is bounded by
\begin{align}\label{cgu}
2\sum_{j=1}^{T}  \frac{C}{|j|\gs\sqrt n} \gL_{k+3}n^{-\xki}
\le 
C \frac{\log n}{\gs\sqrt n} \gL_{k+3}n^{-\xki};
\end{align}
this is the reason for our assumption $\log n\le \gs \sqrt n$, which 
leads to an estimate $C \gL_{k+3}n^{-\xki}$ for \eqref{cgs} too.
The remaining terms give no problems.
\end{proof}

\section{Expansions for the binomial probabilities}\label{Sex}

Taking $x=0$ in \eqref{tcge}--\eqref{tesseen}, we obtain
\begin{align}\label{q1}
  \P(S_n\le n\mu) = 
\HH_{n,k}(0)+\sum_{\ell=1}^k(-1)^{\ell-1}n^{-\ell/2}\gs^{-\ell}
  \psi_\ell\xpar{n\mu}\ddxxx{\ell}{\HH_{n,k-\ell}}(0)+\tR_{n,k}(0)
.\end{align}
If furthermore $n\mu$ is an integer, this yields,
using \eqref{ber} and  $B_0:=1$, and
for convenience defining $Q_0(x):=\Phi(x)$, 
\begin{align}\label{q2}
  \P(S_n\le n\mu)& = 
\HH_{n,k}(0)+\sum_{\ell=1}^k(-1)^{\ell-1}n^{-\ell/2}\gs^{-\ell}
  \psi_\ell\xpar{0}\ddxxx{\ell}{\HH_{n,k-\ell}}(0)+\tR_{n,k}(0)
\notag\\&
=\sum_{\ell=0}^k\sum_{j=0}^{k-\ell}n^{-(j+\ell)/2}\gs^{-\ell}\frac{(-1)^\ell B_\ell}{\ell!}
\ddxxx{\ell}{Q_j}(0)+\tR_{n,k}(0)
\notag\\&
=\frac12+\sum_{m=1}^kn^{-m/2}\sum_{\ell=0}^{m}
\gs^{-\ell}\frac{(-1)^\ell B_\ell}{\ell!}
\ddxxx{\ell}{Q_{m-\ell}}(0)+\tR_{n,k}(0)
.\end{align}
Since $\phi(x)$ is an even function, all its odd derivatives vanish at 0, and
since $\pox_{jr}=0$ unless $j\equiv r\pmod2$, it follows from \eqref{a5a}
that $\ddxxx \ell Q_j(0)=0$ when $j\equiv \ell\pmod2$ (except when $j=\ell=0$).
Hence, all terms in
\eqref{q2} with $m$ even vanish.

\begin{remark}\label{R<}
For $\P(S_n<n\mu)$, we have the same formulas 
with $\psi_1(x)$ replaced by its  left-continuous version $\psi_1(x-)$. 
(All other appearing functions are continuous.)
In \eqref{q2}, this means that the sign is changed for the terms with
$\ell=1$;
all other terms remain the same.
\end{remark}

We specialize to $X\sim\Be(p)$, with $0<p<1$, so $S_n\sim\Bi(n,p)$.
All moments exist, and we have $\gss:=p(1-p)$.
Furthermore, $\gb_j:=\E|X-p|^j\le \gb_2=\gss$; hence, recalling \eqref{a0}
and \eqref{b2},
\begin{align}\label{e1}
  |\gl_j| \le C_j \gL_j \le C_j \gs^{2-j}.
\end{align}
The condition \eqref{te3a} holds with $c=1$ for all $p$.
Hence \refT{TE3} applies provided
\begin{align}\label{e2}
  \sqrt{np(1-p)} \ge\log n,
\end{align}
and then yields, together with \refR{Rsimpler} and using \eqref{e1},
\begin{align}\label{e3}
  |R_{n,k}(x)|,
  |\tR_{n,k}(x)|
\le C_k \gL_{k+3} n^{-\xki}
\le \frac{C_k }{(n\gss)^{(k+1)/2}}
= \frac{C_k }{(np(1-p))^{(k+1)/2}}
\end{align}
We use \eqref{q2} with $k=4$ and obtain, when $np$ is an integer,
\begin{align}\label{e4}
  \P(S_n\le np)
=\frac12+h_1(p)n\qqw + h_3(p) n^{-3/2} + O\bigpar{(np(1-p))^{-5/2}},
\end{align}
where, after some calculations using \eqref{a5b} and \eqref{a2} after finding
the cumulants $\gam_3,\gam_4,\gam_5$,
\begin{align}\label{h1}
  h_1(p)&:=Q_1(0)+\frac{1}{2\gs}Q_0'(0)=
\frac{\gam_3}{6\sqrppi \gs^3} + \frac{1}{2\sqrppi\gs}
=\frac{2-p}{3\sqrppi\sqrt{p(1-p)}},
\\\label{h3}
h_3(p)&:=
Q_3(0)+\frac{1}{2\gs}Q_2'(0)+\frac{1}{12\gss}Q_1''(0)
=
\frac{(2-p)(p^2+23p-23)}{540\sqrppi(p(1-p))^{3/2}}.
\end{align}
Recall that there is no $h_2(p)$ or $h_4(p)$; the corresponding terms
vanish as said above.


\section{Proof of \refT{T1}}\label{SpfT1}
We now prove our main result.

\begin{proof}[Proof of \refT{T1}]
The main idea of the proof is to estimate $q_{m+1}-q_m$ using the estimates
above, in particular \eqref{e4},
but the details will differ for different ranges of $m$.
We write $p_m:=m/n$. We sometimes tacitly assume that $n$ is large enough.
$C_1,C_2,C_3$ will denote some large constants.

Recall $h_1(p)$ and $h_3(p)$ given in \eqref{h1}--\eqref{h3}. A simple
differentiation yields
\begin{align}\label{k1}
  h_1'(p)
=\frac{3p-2}{6\sqrppi\xpar{p(1-p)}^{3/2}}.
\end{align}
Note that $h_1'(p)=0$ for $p=2/3$, 
with $h'(p)<0$ for $0<p<2/3$ and $h'(p)>0$ for $2/3<p<1$; this is the
fundamental reason for the behaviour shown in \refT{T1}, although we also
have to treat error terms.

There is no need to calculate $h_3'(p)$ exactly; it suffices to note from
\eqref{h3} that
\begin{align}\label{k2}
  h_3'(p)= O\bigpar{(p(1-p))^{-5/2}}.
\end{align}

We treat several cases  separately. 

\pfcase{$2\log^2 n \le m  \le n-2\log^2 n$}
Both $p_m$ and $p_{m+1}$ satisfy \eqref{e2}; hence \refT{TE3} applies
and yields \eqref{e4} for both.
By subtraction and the mean value theorem, we
obtain, for some $p_m',p_m''\in[p_m,p_{m+1}]$,
recalling that $p_{m+1}-p_m=n\qw$ and using \eqref{k1}--\eqref{k2},
\begin{align}
q_{m+1}-q_m
&=
h_1'(p_m')n^{-3/2} + h_3'(p_m'') n^{-5/2} + O\bigpar{(np_m)^{-5/2}}
\label{k3-}
\\&   \label{k3}
=\frac{3p_m-2+O(1/n)}{6\sqrppi}(np_m'(1-p_m'))^{-3/2} + O\bigpar{(np_m)^{-5/2}}.
\end{align}


\pfsubcase{$2\log^2 n\le m\le n/2$} 
In this subcase, \eqref{k3} yields
\begin{align}
  \label{k3a}
q_{m+1}-q_m
&
\le
-c(np_m)^{-3/2} + O\bigpar{(np_m)^{-5/2}}
<0,
\end{align}
provided $n$ and thus $np_m=m$ is large enough.

\pfsubcase{$ n/2 \le m \le 2n/3-\CCname{\CCa}$}
In this subcase,
\eqref{k3} similarly yields
\begin{align}
  \label{k3b}
q_{m+1}-q_m
&
\le
-\frac{3\CCa}{6\sqrppi} c n^{-5/2}  + O\bigpar{n^{-5/2}}
<0,
\end{align}
provided $\CCa$ is chosen large enough.

\pfsubcase{$ 2n/3+\CCname{\CCb}\le m \le n-2\log^2 n$}
Similar arguments as in the two preceding subcases yield $q_{m+1}-q_m>0$.

\pfsubcase{$ 2n/3-\CCa\le m \le 2n/3+\CCb$}
This is the most delicate case, since $p_m-2/3=O(1/n)$ and the three 
terms in \eqref{k3-} all are of the same order.
We thus expand one step further and use \refT{TE3} with $k=6$; this yields,
again using \eqref{q2},
\begin{align}
\label{k4}
q_m=\frac12+h_1(p_m)n\qqw + h_3(p_m)n^{-3/2}+h_5(p_m)n^{-5/2}+ O\bigpar{n^{-7/2}},
\end{align}
where we note that $h_5(p)$ is a differentiable function of $p$.
(It can easily be calculated, but we do not need this.)
Taylor expansions yield, recalling $h_1'(\jjj)=0$,
\begin{align}
\label{k5}
q_m&
=\frac12+h_1\ppjjj n\qqw 
+\frac12h_1''\ppjjj\bigpar{p_m-\jjj}^2n\qqw 
+h_3\ppjjj n^{-3/2}
\notag\\&\qquad
+ h_3'\ppjjj\bigpar{p_m-\jjj}n^{-3/2}
+h_5\ppjjj n^{-5/2}
+ O\bigpar{n^{-7/2}}
\notag\\&
=G(n) + H\bigpar{m-2n/3} n^{-5/2} + O\bigpar{n^{-7/2}},
\end{align}
where we define
\begin{align}\label{kg}
  G(n)&:=
\frac12+h_1\ppjjj n\qqw 
+h_3\ppjjj n^{-3/2}
+h_5\ppjjj n^{-5/2},
\\\label{kh}
H(x)&:=
\frac12h_1''\ppjjj x^2
+ h_3'\ppjjj x
.\end{align}
The formula \eqref{k5} holds for $m+1$ too, and thus
\begin{align}\label{k6}
  q_{m+1}-q_m&=\bigpar{H(m+1-2n/3)-H(m-2n/3)}n^{-5/2}+O\bigpar{n^{-7/2}}
\notag\\&=
\Bigpar{h_1''\ppjjj \bigpar{m-\frac{2n}{3}+\frac12}
+h_3'\ppjjj}n^{-5/2}+O\bigpar{n^{-7/2}}
\end{align}
A calculation yields 
\begin{align}
  h_1''(2/3)&=\frac{27}{8\sqrt\pi},
\\
h_3'(2/3)&=\frac{9}{40\sqrt\pi}
=\frac{1}{15}h_1''(2/3).
\end{align}
Since the ratio $1/15<1/6$, it follows from \eqref{k6} that if 
$m \le (2n-2)/3$, then $q_{m+1}-q_m<0$, and  if 
$m \ge (2n-1)/3$, then $q_{m+1}-q_m>0$, for large $n$.

\pfcase{$m<2\log^2 n$}
As said in the introduction, \citet[(29)]{RigolletT} showed that (for every $n$)
$q_{m-1}\ge q_m$, and their proof actually gives $q_{m-1}>q_m$,
for $m\le n/2$.
Alternatively (for large $n$), we can argue using Poisson approximation as
in the next case; we omit the details.

\pfcase{$m>n-2\log^2 n$}
Define $q_m':=\P\bigpar{\Bi(n,m/n) < m}$.
By symmetry, $1-q_m = q'_{n-m}$; hence the claim $q_m<q_{m+1}$ is
equivalent to $q'_{m-1}<q'_m$ for $m<2\log^2 n$.

We use Poisson approximation of the binomial distribution.
It is well-known, see \eg{} \cite[Theorem 2.M]{SJI} that the total variation
distance between $\Binp$ and $\Po(np)$ is less than $p$, and thus,
in particular, 
\begin{align}\label{r1}
\bigabs{q'_m-\P\bigpar{\Po(m)<m}}
\le \dtv\bigpar{\Bi(n,p_m),\Po(np_m)}
<p_m.
\end{align}
We estimate $\P\bigpar{\Po(m)<m}$ by \refT{TE3} (or \refT{TE1}) applied to
$X\sim\Po(1)$. This yields, using \eqref{q2} and \refR{R<},
\begin{align}\label{r2}
    \P\bigpar{\Po(m)<m} = \frac12 -\frac{1}{3\sqrppi} m\qqw
  -\frac{23}{270\sqrppi} m^{-3/2} + O\bigpar{m^{-5/2}}.
\end{align}
Combining \eqref{r1} and \eqref{r2} we find, for $m\ge2$,
\begin{align}
  q_m'-q'_{m-1} =\frac{1}{6\sqrppi} m^{-3/2} + O\bigpar{m^{-5/2}} + O(m/n).
\end{align}
This shows that $q'_m>q'_{m-1}$ for $\CCname{\CCc} <m <2\log^2 n$  
and $n$ large. (In fact, for $\CCc<m <cn^{2/5}$.)

In the remaining subcase $m\le \CCc$, $q_{m-1}'<q'_{m}$ follows from \eqref{r1}
and \refL{LPo} below. 

These cases cover all $m$, which completes the proof of \refT{T1}.
\end{proof}

The proof used the following lemma, of independent interest.
It gives two Poisson versions of the inequality 
mentioned above
for the binomial distribution 
shown by \citet{RigolletT}.
We use their method of proof. 

\begin{lemma}\label{LPo}
For every integer $m\ge0$,
\begin{align}\label{lp1}
  \P\bigpar{\Po(m)<m} &< \P\bigpar{\Po(m+1)<m+1}<\frac12,
\\\label{lp2}
  \P\bigpar{\Po(m)\le m} &> \P\bigpar{\Po(m+1)\le m+1}>\frac12
.\end{align}
\end{lemma}

\begin{proof}
We may assume $m\ge1$.
  Consider a Poisson process with intensity 1 on $(0,\infty)$ and denote its
  points by $T_1<T_2<\dots$.
The number of points in $[0,m]$ is $\Po(m)$, and thus $\Po(m)<m\iff T_m>m$.
The distribution of $T_m$ is $\gG(m)$, and thus
\begin{align}
  \label{t1}
\P\bigpar{\Po(m)<m} = \P(T_m>m)
=
\int_m^\infty\frac{t^{m-1}}{(m-1)!}e^{-t}\dd t.
\end{align}
Hence,
\begin{align}
  \label{t2}
\qquad&\hskip-2em \P\bigpar{\Po(m)<m} - \P\bigpar{\Po(m+1)<m+1}
\notag\\&
=
\int_m^\infty\frac{t^{m-1}}{(m-1)!}e^{-t}\dd t
- \int_{m+1}^\infty\frac{t^{m}}{m!}e^{-t}\dd t
\notag\\&
=
\int_m^\infty\lrpar{\frac{t^{m-1}}{(m-1)!}e^{-t}-\frac{t^m}{m!}e^{-t}}\dd t
+ \int_m^{m+1}\frac{t^{m}}{m!}e^{-t}\dd t
\notag\\&
=
\Bigsqpar{\frac{t^m}{m!}e^{-t}}_m^\infty
+ \int_m^{m+1}\frac{t^{m}}{m!}e^{-t}\dd t
=
\int_m^{m+1}\frac{t^{m}e^{-t}-m^me^{-m}}{m!}\dd t<0,
\end{align}
since $t^m e^{-t}$ is decreasing for $t\ge m$.
Similarly,
\begin{align}
  \label{t3}
\qquad&\hskip-2em \P\bigpar{\Po(m)\le m} - \P\bigpar{\Po(m+1)\le m+1}
\notag\\&
=
\int_m^\infty\frac{t^{m}}{m!}e^{-t}\dd t
- \int_{m+1}^\infty\frac{t^{m+1}}{(m+1)!}e^{-t}\dd t
\notag\\&
=
\int_m^\infty\lrpar{\frac{t^{m}}{m!}e^{-t}-\frac{t^{m+1}}{(m+1)!}e^{-t}}\dd t
+ \int_m^{m+1}\frac{t^{m+1}}{(m+1)!}e^{-t}\dd t
\notag\\&
=
\int_m^{m+1}\frac{t^{m+1}e^{-t}-m^{m+1}e^{-m}}{(m+1)!}\dd t>0,
\end{align}
since $t^{m+1} e^{-t}$ is increasing for $t\le m+1$.
\end{proof}

\newcommand\AAP{\emph{Adv. Appl. Probab.} }
\newcommand\JAP{\emph{J. Appl. Probab.} }
\newcommand\JAMS{\emph{J. \AMS} }
\newcommand\MAMS{\emph{Memoirs \AMS} }
\newcommand\PAMS{\emph{Proc. \AMS} }
\newcommand\TAMS{\emph{Trans. \AMS} }
\newcommand\AnnMS{\emph{Ann. Math. Statist.} }
\newcommand\AnnPr{\emph{Ann. Probab.} }
\newcommand\CPC{\emph{Combin. Probab. Comput.} }
\newcommand\JMAA{\emph{J. Math. Anal. Appl.} }
\newcommand\RSA{\emph{Random Structures Algorithms} }
\newcommand\DMTCS{\jour{Discr. Math. Theor. Comput. Sci.} }

\newcommand\AMS{Amer. Math. Soc.}
\newcommand\Springer{Springer-Verlag}
\newcommand\Wiley{Wiley}

\newcommand\vol{\textbf}
\newcommand\jour{\emph}
\newcommand\book{\emph}
\newcommand\inbook{\emph}
\def\no#1#2,{\unskip#2, no. #1,} 
\newcommand\toappear{\unskip, to appear}

\newcommand\arxiv[1]{\texttt{arXiv}:#1}
\newcommand\arXiv{\arxiv}

\def\nobibitem#1\par{}

\section*{Acknowledgement}
This work is dedicated to the memory of my teacher and colleague Carl-Gustav
Esseen (1918--2001).
I thank \Vasek{} \Chvatal{} for introducing me to this problem,
and Xing Shi Cai for finding two errors in a previous version.

\end{document}